\documentclass[reqno]{amsart}
\usepackage[left=1in,right=1in,top=1in,bottom=1in]{geometry}
\usepackage{tikz}
\usetikzlibrary{shapes,decorations,calc,arrows}
\usepackage[foot]{amsaddr}
\usepackage{amsthm}
\usepackage{amscd}
\usepackage{amsfonts}
\usepackage{amsmath}
\usepackage{amssymb}
\usepackage{mathrsfs}
\usepackage{multirow}
\usepackage{verbatim}
\usepackage{url}
\usepackage[hidelinks]{hyperref}
\usepackage{graphicx}
\usepackage{cite}
\usepackage{fancyhdr}
\usepackage{yfonts}
\usepackage{setspace}
\usepackage{titlesec}
\usepackage{enumitem}

\usepackage{ dsfont }

\usepackage{tocloft} 

\titleformat{\section}[hang]
{\normalfont\Large\bfseries}
{\thesection.}{0.5em}{}

\titlespacing*{\section}{0pc}{2pc}{0.25pc}

\titleformat{\subsection}[runin]
{\normalfont\large\bfseries}
{\thesubsection}{0.5em}{}

\titlespacing{\subsection}{0pc}{1.5pc}{0.5pc}

\newcommand{\N}{\mathbb{N}}
\newcommand{\Z}{\mathbb{Z}}

\newcommand{\R}{\mathbb{R}}

\renewcommand{\l}{\ell}

\newtheorem{thm}{Theorem}[section]

\newtheorem{prop}[thm]{Proposition}
\newtheorem{lem}[thm]{Lemma}
\newtheorem*{lem*}{Lemma}
\newtheorem{cor}[thm]{Corollary}

\theoremstyle{definition}
\newtheorem{defi}[thm]{Definition}
\newtheorem{ex}[thm]{Example}

\title{\textbf{Translation Actions on Non-Unimodular Groups and Strong Ergodicity}}
\author{FEHMİ EKİN GİRİTLİOĞLU}
\address{Department of Mathematics, Michigan State University\hfill \url{giritlio@msu.edu}}
\date{}

\begin{document}

\begin{abstract}
We investigate translation actions of countable dense subgroups of non-unimodular locally compact second countable (lcsc) groups. 
Using left-right actions, we show that the left translation action $\Gamma \curvearrowright G$ given by a countable dense subgroup $\Gamma$ of a locally compact second countable group $G$ can only be strongly ergodic if $G$ is almost unimodular. We show that the strong ergodicity of the action $\Gamma \curvearrowright G$ for an almost unimodular lcsc group $G$ is equivalent to the strong ergodicity of $\Gamma \cap \ker(\Delta_G)\curvearrowright \ker(\Delta_G)$, where $\Delta_G$ is the modular function. We demonstrate the absence of rigidity, by showing that non-isomorphic lcsc almost unimodular groups can admit orbit equivalent translation actions.
\end{abstract}

\maketitle


\section*{Introduction}

Defined first for actions of countable groups on a measure space \cite{Sc80}, a non-singular countable measurable equivalence relation $\mathcal{R}$ on a standard probability space $(X,\mu)$ is called \textit{strongly ergodic} if the relation does not contain any non-trivial almost invariant sequences. An almost invariant sequence is a sequence of measurable subsets $A_n$ such that for all $g \in [\mathcal{R}]$ (the full group of the equivalence relation), $\mu(A_n \Delta gA_n) \rightarrow 0$. A sequence $A_n$ is called trivial if $\mu(A_n)(1-\mu(A_n)) \rightarrow 0$. Although not a complete match, strong ergodicity is closely tied to fullness of von Neumann algebras. Houdayer-Marrakchi-Verraedt have shown that Connes' structure theorem for almost periodic weights on full von Neumann algebras \cite[Lemma 4.8]{Co74} has an analogue for almost periodic measures on strongly ergodic equivalence relations \cite[Theorem D]{HMV19}. In Section 1, we use their theorem to prove that if the left translation action of a countable dense subgroup $\Gamma$ of a lcsc group $G$ is strongly ergodic, then $G$ is almost unimodular. Following \cite{gg25}, we call a locally compact group $G$ with modular function $\Delta_G:G \rightarrow \R_+$ \textit{almost unimodular} if $\ker\Delta_G$ is open. In the second countable case, this is equivalent to $\Delta_G(G) \neq \R_+$. We prove a slightly generalized version of the following in Theorem \ref{strergau}:
\begingroup
\def\thethm{A}
\addtocounter{thm}{-1}

\begin{thm}
    Suppose $G$ is a lcsc group such that $\Delta_G(G) = \R_+$. Then, no left translation action of a countable dense subgroup on $G$ equipped with its left Haar measure is strongly ergodic.

\label{thmA}
\end{thm}
\endgroup
\noindent For a non-unimodular, almost unimodular group $G$, given a countable dense subgroup $\Gamma$, it turns out that the strong ergodicity of $\Gamma \curvearrowright G$ is equivalent to the strong ergodicity of $\Gamma \cap \ker \Delta_G \curvearrowright \ker(\Delta_G)$. This reduces the question of strong ergodicity for left actions to the unimodular case.

For strongly ergodic left translation actions $\Gamma \curvearrowright G$, rigidity phenomena have been observed (see \cite{Fu11,Io17}). The action $\Gamma \curvearrowright G$ may ``remember'' the inclusion $\Gamma < G$ in the sense that orbit equivalent actions $\Gamma \curvearrowright G$ and $\Lambda \curvearrowright H$ imply the existence of a topological isomorphism $\delta : G \rightarrow H$ satisfying $\delta(\Gamma) = \Lambda$. The current most general rigidity statement for non-compact groups is due to Ioana \cite[Theorem A]{Io17}, which asserts that when $G$ is simply connected, a strongly ergodic action $\Gamma \curvearrowright G$ remembers the inclusion $\Gamma < G$ in this way. Notice that $G$ being simply connected and admitting a strongly ergodic left translation action forces $G$ to be unimodular. Given a non-unimodular almost unimodular lcsc group $G$, we show in Section 2 that for any strongly ergodic action by a countable dense subgroup $\Gamma \curvearrowright G$, there exists an orbit equivalent action $\Lambda \curvearrowright H$ such that $\Delta_G(G) \neq \Delta_H(H)$. However, in the case where the unimodular part $G_1$ of $G$ is simply connected, the action $\Gamma \curvearrowright G$ remembers the inclusion $\Gamma \cap \ker\Delta_G <\ker\Delta_G$, in the sense that if $\Gamma \curvearrowright G$ and $\Lambda \curvearrowright H$ are stably orbit equivalent, then there exists a topological isomorphism $\delta$ from $\ker\Delta_G$ to $\ker\Delta_H$ satisfying $\delta(\Gamma \cap \ker\Delta_G\big) = \Lambda \cap \ker\Delta_H$.

In Section 3, we produce examples of non-injective non-full von Neumann algebras. Using translations and combining left and right actions on non-almost unimodular groups, We construct non-injective non-full factors of types II$_\infty$ and III$_\lambda$ for $\lambda \in (0,1]$. Moreover, using strongly ergodic actions on almost unimodular groups, we construct translation actions whose equivalence relations $\mathcal{R}$ have prescribed Sd($\mathcal{R}$).

\section*{Acknowledgements}
We'd like to sincerely thank our advisor Brent Nelson for their encouragement and tireless support throughout the preperation of this paper.

\section{Strong Ergodicity, the Non-Unimodular Case}

We recall certain notions of orbit equivalence. All actions are assumed to be non-singular.

    Given two actions $\Gamma \curvearrowright (G,\mu)$ and $\Lambda \curvearrowright (H,\lambda)$, the actions are called \textit{orbit equivalent} (OE) if there exists a non-singular Borel isomorphism of measure spaces $\theta: G \to H$ such that for almost every $x \in G$, ``$\theta$ sends orbits to orbits":
    \begin{equation*}
        \theta(\Gamma x) = \Lambda \theta(x).
    \end{equation*}
    The two actions are called \textit{stably orbit equivalent} (SOE) if there are $E \subset G$, $F \subset H$ two non-null sets such that for almost every $x \in G$, there exists a Borel isomorphism of measure spaces $\theta: E \to F$ such that
    \begin{equation*}
        \theta(\Gamma x \cap E) = \Lambda\theta(x) \cap F
    \end{equation*}

To make our statements more concise, we define the notion of a \textit{countable dense inclusion}.

\begin{defi}
For a lcsc group $G$, denote by $m_G$ its left Haar measure and $\mu_G$ a probability measure on $G$ equivalent to $m_G$. Let $\Gamma$ be a countable dense subgroup of $G$. The inclusion $\Gamma < G$ is called a \textit{countable dense inclusion}. The left translation action of $\Gamma$ on the measure space $(G,\mu_G)$ is called a \textit{countable dense action}. 
\end{defi}

Countable dense actions preserve the possibly infinite measure $m_G$. It is likely known to experts that countable dense actions are ergodic. For the convenience of the online researcher, we provide an argument. Let $L_gf(x) = f(g^{-1}x)$ for $g \in G$. Then, for any $f \in C_c(G)$ (compactly supported continuous function) and any $A \subset X$, the map $T_f \colon g \mapsto \int_G L_gf(x) 1_A (x) dm_G(x)$ is continuous \cite[Proposition 2.42]{foAH}. If $A$ is invariant under the countable dense action, $T$ is constant for all $f\in C_c(G)$. Therefore if $A$ is non-null, $1_A dm_G$ is a left Haar measure \cite[Proposition 11.4]{foRA}. Hence $A = X$ $m_G$-almost everywhere, equivalently $\mu_G$-almost everywhere, which implies ergodicity.

We will prove a slightly more general version of Theorem \ref{thmA}. Let us recall the Radon-Nikodym cocyle. Let $\mathcal{R}$ be a countable Borel equivalence relation on a standard probability space $(X,\mu)$. On $X \times X$-Borel subsets of $\mathcal{R}$, define the measures $\nu_\ell(E) := \int_X \sum_{y \sim x} 1_E (x,y) d\mu(x)$ and $\nu_r(E) := \int_X \sum_{x \sim y} 1_E (x,y) d\mu(y)$. The \textit{Radon-Nikodym cocycle} $\delta_{\mu}$ is $\frac{d\nu_\ell}{d\nu_r}$.

\begin{thm} \label{strergau}
    Suppose $G$ is a non-almost unimodular lcsc group. Suppose $\Gamma_0$ and $\Pi_0$ are two countable subgroups of $G$. Then, the equivalence relation $\mathcal{R}$ on $(G, \mu_G)$ given by
    \begin{equation*}
        x \sim \gamma x \pi \ \ \ \  \gamma \in \Gamma_0, \pi \in \Pi_0
    \end{equation*}
    is not strongly ergodic.
\end{thm}

\begin{proof}
    Denote by $\Delta_G : G \rightarrow \R_+$ the modular homomorphism. Since $G$ is non-almost unimodular, one has
        \begin{equation*}
        \Delta_G(G) = \R_+,   
    \end{equation*}
    (see the discussion following \cite[Definition 2.2]{gg25}).
    Let $\Gamma$ and $\Pi$ be countable dense subgroups of $G$ containing $\Gamma_0$ and $\Pi_0$ respectively. Since $\Gamma$ and $\Pi$ are countable, there exists $g \in G$ such that $\Delta_G(g) \notin \Delta_G(\Gamma \vee \Pi)$, where $\vee$ denotes the generated subgroup. Let $\Lambda := \Pi \vee g $. Notice $\Lambda$ is also a countable dense subgroup and $\Delta_G(\Gamma) \subsetneq \Delta_G(\Lambda)$.

    Consider the left-right action equivalence relation $\overline{\mathcal{R}}$ on $(G,\mu_g)$ given by 
    \begin{equation*}
        x \sim \gamma x \lambda \ \ \ \  \gamma \in \Gamma, \lambda \in \Lambda
    \end{equation*}
    We get that the Radon-Nikodym cocycle with respect to $m_G$ is
    \begin{equation*}
        \delta_{m_G}(x,\gamma x \lambda) = \frac{1}{\Delta_G(\lambda)}
    \end{equation*}
    and for the right Haar measure $\rho_G$:
    \begin{equation*}
        \delta_{\rho_G} (x,\gamma x \lambda) = \frac{1}{\Delta_G(\gamma)}
    \end{equation*}
    Therefore both left and right Haar measures are almost periodic for $\overline{\mathcal{R}}$. Now suppose $\mathcal{R}$ is strongly ergodic. Then $\overline{\mathcal{R}}$ is strongly ergodic, and by \cite[Theorem 5.3]{HMV19},
    \begin{equation*}
        \bigcap_{U \subset G \text{ Borel, } U \neq \emptyset} \delta_{m_G}\big(\mathcal{\overline{R}} \cap  ( U \times U)\big) = \bigcap_{U \subset G \text{ Borel, } U \neq \emptyset} \delta_{\rho_G}\big(\mathcal{\overline{R}} \cap  ( U \times U)\big)
    \end{equation*}
    Notice $\ker(\delta_{m_G})$ contains the countable dense action $\Gamma \curvearrowright G$ and hence is ergodic. Similarly, $\ker(\delta_{\rho_G})$ is ergodic. Therefore by the proof of \cite[Lemma 5.6]{HMV19}, $\delta_{m_G}\big(\mathcal{\overline{R}} \cap  ( U \times U)\big)$ is the same subset of $\R_+$ for all non-null $U$. The same holds for $\delta_{\rho_G}$. We obtain
    \begin{equation*}
    \Delta_G(\Lambda) = \delta_{m_G}(\mathcal{\overline{R}}) = \delta_{\rho_G}(\overline{\mathcal{R}}) = \Delta_G(\Gamma),
    \end{equation*}
    which is absurd.
\end{proof}

We have established that only almost unimodular lcsc groups can admit strongly ergodic countable dense actions. Let us now consider non-unimodular almost unimodular groups. As a measure space, such groups are countably many copies of their unimodular part. This fact can be used to show that whether a countable dense action on such a group is strongly ergodic can be gleaned from its unimodular part.

\begin{thm}
    Suppose $G$ is an almost unimodular lcsc group. Let $\Gamma$ be a countable dense subgroup of $G$. The following are equivalent:
    \begin{enumerate}[label=(\roman*),ref=(\roman*)]
        \item The left translation action of $\Gamma$ on $G$ is strongly ergodic.
        \item The left translation action of $\Gamma_1 := \Gamma \cap \ker\Delta_G$ on $G_1 := \ker\Delta_G$ is strongly ergodic.
    \end{enumerate}
\end{thm}

\begin{proof}
    $(i) \Rightarrow (ii)$: Suppose the action of $\Gamma_1$ on $G_1$ contains a non-trivial almost invariant sequence $A_n$. For each $\lambda$, choose $g_\lambda \in \Gamma \cap \Delta_G^{-1}(\{\lambda\})$. Such elements exist by the density of $\Gamma$ requiring $\Delta_G(\Gamma) = \Delta_G(G)$. Define
    \begin{equation*}
        B_n :=\bigcup_{\lambda \in \Delta_G(G)}   g_\lambda A_n.        
    \end{equation*}
    We will show that $B_n$ is an almost invariant sequence for the countable dense action $\Gamma \curvearrowright G$. Given $\epsilon>0$, let $I=\{\lambda_1, \ldots, \lambda_n\} \subset \Delta_G(G)$ such that $\mu_G(\Delta_G^{-1}(I)) > 1 - \epsilon/2$. Given $s \in \Gamma_1$ and $\lambda \in \Delta_G(G)$, notice 
    \begin{equation*}
        [(sg_\lambda B_n) \Delta B_n ]\cap \Delta_G^{-1}(\{\lambda_i\})=g_{\lambda_i}(A_n \Delta g_{\lambda_i}^{-1} s  g_\lambda g_{\lambda^{-1}{\lambda_i}}A_n)
    \end{equation*}
    Notice $g_{\lambda_i}^{-1} s  g_\lambda g_{\lambda^{-1}{\lambda_i}} \in \Gamma_1$. Using $A_n$ is almost invariant for $\Gamma_1$ acting on $G_1$, that $I$ is finite, that $(g_\lambda^{-1})_* \mu_G$ is absolutely continuous with respect to the finite measure $\mu_G$, one can find $m$ large enough such that
    \begin{equation*}
        \sum_{\lambda_i \in I}\mu_G\big(g_{\lambda_i}(A_m \Delta g_{\lambda_i}^{-1} s  g_\lambda g_{\lambda^{-1}{\lambda_i}}A_m)\big) < \epsilon/2.
    \end{equation*}
We have
    \begin{align*}
        \mu_G(sg_\lambda B_m \Delta B_m) &= \mu_G\big((sg_\lambda B_m \Delta B_m) \cap \Delta_G^{-1}(I)\big) + \mu_G\big(G \setminus \Delta_G^{-1}(I)\big) \\
        &< \sum_{\lambda_i \in I}\mu_G\big(g_{\lambda_i}(A_m \Delta g_{\lambda_i}^{-1} s  g_\lambda g_{\lambda^{-1}{\lambda_i}}A_m)\big) + \epsilon/2 \\
        &< \epsilon
    \end{align*}
    Hence $B_n$ is an almost invariant sequence for $\Gamma$ acting on $G$. Note that $B_n$ is non-trivial because $A_n$ is non-trivial.\\

    \noindent $(ii) \Rightarrow (i)$: Suppose the action of $\Gamma$ on $G$ contains a non-trivial almost invariant sequence $A_n$. Then, for each $\lambda \in \Delta_G(G)$, the sequence $B_n^\lambda := A_n \cap \Delta_G^{-1}(\{\lambda\})$ is a possibly trivial almost invariant sequence for the countable dense action $\Gamma_1 \curvearrowright \Delta_G^{-1}(\{\lambda\})$. Since $A_n$ is non trivial, $\mu_G(B_n^\lambda)$ can't converge to $\mu_G(\Delta_G^{-1}(\{\lambda\}))$ for all $\lambda$. Similarly, $\mu_G(B_n^\lambda)$ can't converge to $0$ for all $\lambda$. Suppose there exists $\lambda_1$ and $\lambda_2$ such that $\mu_G(B_n^{\lambda_1})$ converges to $\mu_G(\Delta_G^{-1}(\{\lambda_1\}))$ and $\mu_G(B_n^{\lambda_2})$ converges to $0$. Let $h \in \Gamma \cap \Delta_G^{-1}(\{\lambda_2^{-1}\lambda_1\})$. We have:
    \begin{align*}
        \mu_G(hA_n \Delta A_n) \geq \mu_G(h B^{\lambda_1}_n \Delta B^{\lambda_2}_n) \geq \mu_G(hB^{\lambda_1}_n) - \mu_G(B^{\lambda_2}_n) = \mu_G(B^{\lambda_1}_n) - \mu_G(B^{\lambda_2}_n) \rightarrow \mu_G(\Delta_G^{-1}(\{\lambda_1\})) > 0
    \end{align*}
    contradicting the almost invariance of $A_n$. Therefore there exists a $\lambda$ such that  $B_n^\lambda$ is a non-trivial almost invariant sequence for the countable dense action $\Gamma_1 \curvearrowright \Delta_G^{-1}(\{\lambda\})$, which is orbit equivalent (OE) to the action of $\Gamma_1$ on $G_1$ via the map $x \mapsto g_{\lambda^{-1}} x$. The OE-invariance of strong ergodicity completes the proof.
\end{proof}

\section{Absence of Rigidity}

In the non-unimodular case, we demonstrate the absence of rigidity with the following proposition.

\begin{prop} \label{oep}
    Every strongly ergodic left translation action of a countably infinite subgroup $\Gamma$ of a non-unimodular group $G$ is orbit equivalent to $\big(\Gamma \cap \ker(\Delta_G) \big)\times \Z \curvearrowright \ker(\Delta_G) \times \Z$.
\end{prop}

\begin{proof}
    It suffices to show the action is OE to $\big(( \Gamma \cap \ker(\Delta_G)) \times \Delta_G(G) \big) \curvearrowright \ker(\Delta_G) \times \Delta_G(G)$, where $\Delta_G(G)$ is equipped with the discrete topology. The result would then follow from $\Gamma \curvearrowright \Gamma$ being OE to $\Z \curvearrowright \Z$ for any countably infinite discrete group $\Gamma$. \\
    \noindent By Theorem \ref{strergau}, $G$ is almost unimodular and $\Delta_G(G)$ is a countable proper subgroup of $\R_+$. 
    Pick $\gamma_t$ from $\Gamma \cap \Delta_G^{-1}(t)$ for all $t \in \Delta_G(G)$. We claim that the Borel isomorphism $\theta$ defined by $\theta(x) = (\gamma_t^{-1}x,t)$ for $\Delta_G(x) = t$ implements the orbit equivalence. Indeed, for $x \in G$ with $\Delta_G(x) = u$ and $\gamma \in \Gamma$ with $\Delta_G(\gamma) = v$
    \begin{equation*}
    \theta (\gamma x) = (\gamma^{-1}_{uv}\gamma x, uv) =  (\gamma_{uv}^{-1}\gamma\gamma_u,v)(\gamma_u^{-1}x,u) = (\gamma_{uv}^{-1}\gamma\gamma_u,v) \theta(x)
    \end{equation*}
    On the other hand, for $\gamma \in \Gamma \cap \Delta_G(G)$ we have
    \begin{equation*}
        (\gamma,t)\theta(x) = (\gamma \gamma_u^{-1}x,tu) = (\gamma_{tu}^{-1} \gamma^{}_{tu} \gamma \gamma_u^{-1} x,tu) = \theta(\gamma_{tu} \gamma \gamma_u^{-1} x)
    \end{equation*}
    proving that the countable dense action $\Gamma \curvearrowright G$ is OE to $\big(( \Gamma \cap \ker(\Delta_G)) \times \Delta_G(G) \big) \curvearrowright \ker(\Delta_G) \times \Delta_G(G)$. 
\end{proof}

\begin{cor}
    Given a strongly ergodic countable dense action $\Gamma \curvearrowright G$ on a non-unimodular group, there exists at least countably many non topologically isomorphic groups $G_i$ which admit orbit equivalent countable dense actions.
\end{cor}
    
\begin{proof}
    Let $t \in \Delta_G(G) \setminus \ker(\Delta_G)$. Let $g \in \Gamma$ such that $\Delta(g) = t$. For $i \in \N$, let $\Gamma_i$ be the subgroup generated by $g^i$.
    Let $T_i$ be the subgroup of $\R_+$ generated by $t^i$. Let $G_i = G \cap \Delta_G^{-1}(T_i)$. By the previous proposition, for all $i>1$, the countable dense actions $\Gamma_i \curvearrowright G_i$ are all OE to $\Gamma \curvearrowright G$ but $G_i$ is not topologically isomorphic to $G$, since their image under their modular homomorphisms differ.
\end{proof}

Under the assumption of simple connectedness, the ``unimodular part" of the action is still rigid.

\begin{thm}
    Let $G$ be a non-unimodular almost unimodular lcsc group with simply connected unimodular part. Suppose $\Gamma \curvearrowright G$ is a strongly ergodic countable dense action. Let $H$ be an almost unimodular lcsc group with simply connected unimodular part. Let $\Lambda \curvearrowright H$ be a countable, not necessarily dense action.

    Then, $\Gamma \curvearrowright G$ is stably orbit equivalent to $\Lambda \curvearrowright H$ if and only if there exists a topological isomorphism $\delta : \ker\Delta_G \to  \ker\Delta_H$ satisfying $\delta\big(\Gamma \cap \ker\Delta_G\big) = \Lambda \cap \ker\Delta_H$.
\end{thm}

\begin{proof}
    Let $\Gamma_1 = \Gamma \cap \ker\Delta_G$. Similarly denote by $\Lambda_1, G_1, H_1$ their intersections with the kernels of their corresponding modular homomorphisms $\Delta_G$ or $\Delta_H$. By Proposition \ref{oep}, the countable dense actions $\Gamma \curvearrowright G$ and $\Lambda \curvearrowright H$ are OE to $\Gamma_1 \times \Z \curvearrowright G_1 \times \Z$ and $\Lambda_1 \times \Z \curvearrowright H_1 \times \Z$ respectively. 
    
    Suppose $\Gamma \curvearrowright G$ and $\Lambda \curvearrowright H$ are SOE. Then, $\Gamma_1 \times \Z \curvearrowright G_1 \times \Z$ and $\Lambda_1 \times \Z \curvearrowright H_1 \times \Z$ are SOE with a map $\theta : E \to F$. Therefore, there exists $t,s \in \R_+$ such that $E_t := E \cap (G_1 \times \{t\})$ is non-null and $\theta(E_t) = F_s$ for some set satisfying $F_s = F_s \cap (H_1 \times \{s\})$. Notice $\theta$ implements SOE between $\Gamma_1 \curvearrowright G_1$ and $\Lambda_1 \curvearrowright H_1$. Since $G_1$ and $H_1$ are connected, an application of \cite[Theorem A]{Io17} completes the proof.
\end{proof}

\section{von Neumann Algebras}

Using translation actions on a non almost unimodular group, one can make non-full non-injective von Neumann algebras of types II$_\infty$ and III$_\lambda$ for $\lambda \in (0,1]$. See \cite{FMII} for the construction of a von Neumann algebra $L(\mathcal{R})$ from an equivalence relation $\mathcal{R}$.

\begin{lem}
    Any non-amenable countable dense subgroup $\Gamma$ of a connected  lcsc group $G$ induces a non-amenable left translation action.
\end{lem}

\begin{proof}
    The only open subgroup of a connected lcsc group is itself. Therefore $\Gamma$ doesn't satisfy Definition 8.3 of \cite{BG07}, and Theorem 8.4 from the same article shows that the countable dense action $\Gamma \curvearrowright G$ is non-amenable.
\end{proof}

\begin{ex}
    Consider $G = GL_n(\R)^+ \rtimes \R^n$ for some $n\geq2$, the group of orientation preserving affine transformations on $\R^n$. $G$ is connected. Let $\Gamma$ be a countable dense subgroup of $G$ containing $SL_n(\Z)$. Since $SL_2(\Z)$ (and hence $\Gamma$) contains the free group on two generators, $\Gamma$ is not amenable. By the previous lemma and Theorem \ref{strergau}, the countable dense action $\Gamma \curvearrowright G$ is neither amenable nor strongly ergodic. Moreover, $m_G$ indudces an infinite trace. Hence, the von Neumann algebra $L(\Gamma \curvearrowright G)$ is a non-injective non-full factor of type II$_\infty$. Given $t \in \R_+$,  there exists $g \in G$ with $\Delta_G(g) = t$. Let $U : L^2(\mathcal{R}) \to L^2(\mathcal{R})$ be defined as $U(f)(x,y) = U(f)(g^{-1}x,g^{-1}y)$. The action $\alpha_g: x \mapsto UxU^*$ is an infinite trace scaling action with scaling factor $t$. We therefore obtain that the fundamental group of $L(\Gamma \curvearrowright G)$ is $\R^+$. 
\end{ex}

\begin{ex}
    Using the setup of the previous example, let $\Lambda$ be any countable subgroup of $G$. Consider the left-right equivalence relation $\mathcal{R}$ given by $x \sim \gamma x \lambda$ for $\gamma \in \Gamma$ and $\lambda \in \Lambda$. Since $\Gamma \curvearrowright \Lambda$ is a countable dense action, $\ker(\delta_{m_G})$ is ergodic. Therefore by the proof of \cite[Lemma 5.6]{HMV19}, the asymptotic ratio set of $\mathcal{R},m_G$ is the essential range of the Radon-Nikodym cocycle $\delta_{m_G}$, which in this case is $\overline{\Delta_G(\Lambda)}$. By \cite[Theorem XIII.2.13]{Ta3}, $\Delta_G(\Lambda)$ is equal to the modular spectrum $S$ of $L(\mathcal{R})$. This means with an appropriate choice of $\Lambda$, one can obtain a von Neumann algebra $M$ of all subtypes III$_\lambda$ for $\lambda \in (0,1]$. $M$ contains $L(\Gamma \curvearrowright G)$ with expectation and hence is non-injective. Moreover by Theorem \ref{strergau}, the von Neumann algebra is not full.
\end{ex}

The Sd invariant for an equivalence relation was defined by Houdayer-Marrakchi-Verraedt in \cite{HMV19}. One can construct equivalence relations with prescribed Sd($\mathcal{R}$) using strongly ergodic actions on almost unimodular groups.

\begin{ex}
    Let $T$ be any countable subgroup of $\R_+$. For $n\geq 3$, let $H^n_T$ be the group of linear transformations of $\R^n$ generated by $SO(n)$ and $T1$. Let $G^n_T = H^n_T \ltimes \R^n$. Then, $G^n_T$ is an almost unimodular group with $\Delta_{G^n_T}(G^n_T) = T$. Let $\Gamma \curvearrowright SO(n) \ltimes R^n$ be any strongly ergodic countable dense action, which exists by \cite{BI20}. Let $\Lambda = T1 \vee \Gamma$. Consider the equivalence relation $\mathcal{R}$ on $G^n_T$ defined by $x \sim \lambda_1x\lambda_2, \lambda_i \in \Lambda$. By Theorem 1.3, $\mathcal{R}$ is strongly ergodic. Moreover, $\ker(\delta_{m_G})$ is strongly ergodic, hence Sd$(\mathcal{R})= \delta_{m_G}(\mathcal{R}) = T$.
\end{ex}

We conclude with a corollary of Proposition \ref{oep} for the von Neumann algebras generated by $\Gamma \curvearrowright G$ and $(\Gamma \cap \ker(\Delta_G)) \curvearrowright \ker(\Delta_G)$.

\begin{cor}
    For a countable dense action $\Gamma \curvearrowright G$, we have 
    \begin{equation*}
        L(\Gamma \curvearrowright G) \cong  L\big((\Gamma \cap \ker\Delta_G) \curvearrowright \ker\Delta_G\big)
    \end{equation*}
\end{cor}
\begin{proof}
    The statement is trivial when $G$ is unimodular. Suppose not. Then, by Proposition \ref{oep} we have
    \begin{align*}
        L(\Gamma \curvearrowright G) &\cong L\big((\Gamma \cap \ker\Delta_G) \curvearrowright \ker\Delta_G) \otimes L(\Z \curvearrowright \Z) \\
        &\cong L\big((\Gamma \cap \ker\Delta_G) \curvearrowright \ker\Delta_G\big) \otimes B(\l^2(\Z)) \\
        &\cong L\big((\Gamma \cap \ker\Delta_G) \curvearrowright \ker\Delta_G\big)
    \end{align*}
    The last isomorphism follows from $L\big((\Gamma \cap \ker\Delta_G) \curvearrowright \ker\Delta_G \big)$ being a II$_\infty$ factor. Indeed, $\ker\Delta_G$ admits a Haar measure scaling action and hence is non-compact. So ${\mu_G}_{|\ker \Delta_G}$ induces an infinite trace.
\end{proof}


\bibliographystyle{amsalpha}
\bibliography{references}

\end{document}